\pdfoutput=1
\RequirePackage{ifpdf}
\ifpdf 
\documentclass[pdftex]{sigma}
\else
\documentclass{sigma}
\fi

\usepackage{mathrsfs}
\numberwithin{equation}{section}

\newtheorem{Theorem}{Theorem}[section]
\newtheorem{Lemma}[Theorem]{Lemma}
\newtheorem{Proposition}[Theorem]{Proposition}
 { \theoremstyle{definition}
\newtheorem{Definition}[Theorem]{Definition}
\newtheorem{Example}[Theorem]{Example}
\newtheorem{Remark}[Theorem]{Remark} }

\def\x{\hat{\boldsymbol{x}}}
\def\P{\mathbf{P}}

\def\R{\mathbb{R}}
\def\N{\mathbb{N}}
\def\U{\mathbf{U}}
\def\M{\mathbf{M}}
\def\L{\mathscr{L}}
\def\X{\mathbf{X}}
\def\f{\mathbf{f}}
\def\z{\mathbf{z}}
\def\bx{\mathbf{x}}
\def\u{\mathbf{u}}
\def\p{\mathbf{p}}

\begin{document}

\allowdisplaybreaks

\renewcommand{\thefootnote}{$\star$}

\newcommand{\arXivNumber}{1508.06884}

\renewcommand{\PaperNumber}{077}

\FirstPageHeading

\ShortArticleName{Moments and Legendre--Fourier Series for Measures Supported on Curves}

\ArticleName{Moments and Legendre--Fourier Series\\ for Measures Supported on Curves\footnote{This paper is a~contribution to the Special Issue
on Orthogonal Polynomials, Special Functions and Applications.
The full collection is available at \href{http://www.emis.de/journals/SIGMA/OPSFA2015.html}{http://www.emis.de/journals/SIGMA/OPSFA2015.html}}}

\Author{Jean B.~LASSERRE}

\AuthorNameForHeading{J.B.~Lasserre}

\Address{LAAS-CNRS and Institute of Mathematics, University of Toulouse,\\
 7 Avenue du Colonel Roche, BP 54 200, 31031 Toulouse C\'edex 4, France}
\Email{\href{mailto:lasserre@laas.fr}{lasserre@laas.fr}}
\URLaddress{\url{http://homepages.laas.fr/lasserre/}}

\ArticleDates{Received August 28, 2015, in f\/inal form September 26, 2015; Published online September 29, 2015}

\Abstract{Some important problems (e.g., in optimal transport and optimal control)
have a relaxed (or weak) formulation in a space of appropriate measures
which is much easier to solve. However, an optimal solution~$\mu$ of the latter solves the former if and only if
the measure~$\mu$ is supported on a ``trajectory'' $\{(t,x(t))\colon t\in [0,T]\}$
for some measurable function~$x(t)$.
We provide necessary and suf\/f\/icient conditions on moments~$(\gamma_{ij})$ of
a measure $d\mu(x,t)$ on $[0,1]^2$ to ensure that $\mu$ is supported on a trajectory $\{(t,x(t))\colon t\in [0,1]\}$.
Those conditions are stated in terms of Legendre--Fourier coef\/f\/icients ${\mathbf f}_j=({\mathbf f}_j(i))$
associated with some functions $f_j\colon [0,1]\to {\mathbb R}$, $j=1,\ldots$, where each ${\mathbf f}_j$ is
obtained from the moments~$\gamma_{ji}$, $i=0,1,\ldots$, of~$\mu$.}

\Keywords{moment problem; Legendre polynomials; Legendre--Fourier series}

\Classification{42C05; 42C10; 42A16; 44A60}

\renewcommand{\thefootnote}{\arabic{footnote}}
\setcounter{footnote}{0}

\section{Introduction}

This paper is in the line of research concerned with the following issue: {\it which type and how much of information
on the support of a measure can be extracted from its moments} (a research issue
outlined in a {\it Problem session} at the 2013 Oberwolfach  meeting
on {\it Structured Function Systems and Applications}~\cite{ober-report}). In particular,
a highly desirable result is to obtain necessary and/or suf\/f\/icient conditions
on moments of a given measure to ensure that its support has certain geometric properties.
For instance there is a vast literature on the old and classical $L$-moment problem,
which asks for  moment conditions to ensure that the underlying measure~$\mu$ is absolutely continuous with respect to some reference measure~$\nu$,
and with a density in~$L_\infty(\nu)$. See, for instance,~\cite{diaconis,putinar1,putinar2}, more recently~\cite{density}, and the many references therein.

Here we are interested in a  problem that is somehow ``orthogonal'' to the
$L$-moment problem. Namely, we consider the following generic problem: Let $d\mu(x,t)$ be a probability measure on $[0,1]\times [0,1]$. Provide necessary and/or suf\/f\/icient conditions
on the moments of~$\mu$ to ensure that~$\mu$ is singular with respect to the Lebesgue measure~$d(x,t)$ on
$[0,1]^2$. In fact, and more precisely, suppose that:
\begin{itemize}\itemsep=0pt
\item one knows all moments $\gamma_{i}(j)=\int x^i t^j\, d\mu(x,t)$, $i,j=0,1,\ldots$,
of the measure $\mu$, and
\item the marginal of $\mu$ with respect to the ``$t$'' variable is the Lebesgue measure~$dt$ on~$[0,1]$.
\end{itemize}
Then provide necessary and/or suf\/f\/icient conditions on
the moments $(\gamma_i(j))$ of $\mu$ to ensure
that~$\mu$ is supported on a trajectory $\{(t,x(t))\colon t\in [0,1]\}\subset [0,1]^2$,
for some measurable function
$x\colon [0,1]\to [0,1]$.

In contrast to the $L$-moment problem, and to the best of our knowledge,
the above problem stated in this form has not received a lot of attention in the past even though
it is crucial in some important applications (two of them
having motivated our interest).

{\bf Motivation.} In addition of being of independent interest,
this investigation is motivated by at least two important applications:

-- \emph{The mass transfer $($or optimal transport$)$ problem.} In the weak (or relaxed) Monge--Kantorovich formulation of the mass transport problem originally stated by Monge, one searches for a measure $d\mu(x,t)$ with prescribed marginals~$\nu_x$ and $\nu_t$, and which minimizes some cost
functional $\int c(x,t)\, d\mu(x,t)$. However in the original Monge formulation,
ultimately one would like to obtain an optimal solution~$\mu^*$ of the form
$d\mu^*(x,t)=\delta_{x(t)}\,d\nu_t(t)$
for some measurable function $t\mapsto x(t)$ (the transportation plan) and a crucial issue is to provide
conditions for this to happen\footnote{Here $\delta_{z}$ denote the Dirac measure at the point~$z$.}. For more details the interested reader is referred, e.g., to~\cite[pp.~1--5]{villani} and~\cite{cann}.  There exist some characterizations
of the support of an optimal measure for the weak formulation. For instance,
{\it $c$-cyclical monotonicity} relates optimality with the support of solutions, and
more recently~\cite{opt-t} have shown
in the (more general) context of the generalized moment problem that under some weak conditions
the support of optimal solutions is {\it finitely minimal~$/$~$c$-monotone}.
(As def\/ined in~\cite{opt-t} a set~$\Gamma$ is called f\/initely minimal~$/$~$c$-monotone if each f\/inite measure~$\alpha$ concentrated
on f\/initely many atoms of~$\Gamma$ is cost minimizing among its competitors;
in the optimal transport context, a competitor of $\alpha$ is any f\/inite measure
$\alpha'$ with same marginals as~$\alpha$.) For more details the interested reader is referred to~\cite{opt-t} and the references therein. But such a~characterization does not
say when this support is a trajectory.

-- \emph{Deterministic optimal control.}
Using the concept of {\it occupation measures}, a {\it weak formulation} of deterministic optimal control problems
replaces the original control problem with an inf\/inite-dimensional optimization problem~$\mathcal{P}$ on a space of appropriate (occupation) measures on a~Borel space
$\mathscr{X}\times\mathscr{U}\times [0,1]$ with $\mathscr{X}\subset\R^n$, $\mathscr{U}\subset\R^m$.
For more details the interested reader is referred, e.g., to~\cite{sicon, vinter}, and the many references therein.
An important issue is to provide conditions on the problem data under which the optimal value
of the relaxed problem~$\mathcal{P}$ is the same as that of the original problem; see, e.g.,~\cite{vinter}. Again this is the case
if some optimal solution~$\mu^*$ (or every element of a minimizing sequence) of the relaxed problem is such that
every marginal $\mu^*_{j}$ of $\mu^*$ with respect to $(x_j,t)$, $j=1,\ldots,n$,
and every marginal $\mu^*_{\ell}$ of~$\mu^*$ with respect to~$(u_\ell,t)$, $\ell=1,\ldots,m$,
is supported on a~trajectory $\{(t,x_j(t))\colon t\in [0,1]\}$ and on a trajectory $\{(t,u_\ell(t))\colon t\in [0,1]\}$ for some measurable functions
$t\mapsto x_j(t)$ and $t\mapsto u_\ell(t)$ on~$[0,1]$.

{\bf Contribution.}
Of course there is a particular case where one may conclude that~$\mu$ is singular
with respect to the Lebesgue measure on~$[0,1]^2$.
If there is a polynomial $p\in\R[x,t]$ of degree say $d$, such that its vector of coef\/f\/icients~$\p$ is in the kernel of the moment matrix $\M_s$ (where $\M_s[(i,j),(k,\ell)]=\gamma_{i+k,j+\ell}$, $i+j, k+\ell\leq s$, with $d\leq s$),
then~$\mu$ is supported on the variety $\{(x,t)\in [0,1]^2\colon p(x,t)=0\}$ and therefore is singular with respect to the Lebesgue measure on~$[0,1]^2$.  But it may happen that
$p(x,t)=p(y,t)=0$ for some $t$ and some~$x\neq y$ and so even in this case additional conditions are needed to ensure existence of a trajectory $\{(t,x(t))\colon t\in [0,1]\}$.

We provide a set of explicit necessary and suf\/f\/icient conditions
on the moments $\boldsymbol{\gamma}_i=(\gamma_i(j))$ which state that for every f\/ixed~$i$, the moments
$\gamma_i(j)$, $j=0,1,\ldots$, are limits of certain $i$-powers of the moments $\boldsymbol{\gamma}_1$.

More precisely, an explicit linear transformation $\Delta \boldsymbol{\gamma}_1$ of the inf\/inite vector $\boldsymbol{\gamma}_1$
is the vector of (shifted) Legendre--Fourier coef\/f\/icients associated with the function $t\mapsto x(t)$. Then the conditions state that for each f\/ixed $i=2,3,\ldots$, the vector $\Delta \boldsymbol{\gamma}_i$ should be the vector of (shifted) Legendre--Fourier coef\/f\/icients associated with the function $t\mapsto x(t)^i$, which in turn are expressible in
terms of limits of ``$i$-powers'' of coef\/f\/icients of $\Delta\boldsymbol{\gamma}_1$.

At last but not least, it should be noted that all results of this paper are easily transposed to
the multi-dimensional case of a measure $d\mu(\mathbf{x},t)$ on $[0,1]^n\times [0,1]$ and supported on
a trajectory $\{(t,\mathbf{x}(t))\colon t\in [0,1]\}\subset [0,1]^{n+1}$
for some measurable mapping $\mathbf{x}\colon [0,1]\to [0,1]^n$. Indeed by proceeding {\it coordinate-wise}
for each function $t\mapsto x_i(t)$, $i=1,\ldots,n$,
one is reduced to the case $[0,1]^2$ investigated here.

\section[Notation, definitions and preliminary results]{Notation, def\/initions and preliminary results}

\subsection[Notation and definitions]{Notation and def\/initions}

Given a Borel probability measure $\mu$ on $[0,1]^2$, def\/ine
\begin{gather}
\label{def-moment}
\gamma_{i}(j) = \int_{[0,1]^2}t^j x^i\, d\mu(x,t),\qquad i,j=0,1,\ldots,
\end{gather}
and for every f\/ixed $i\in \N$, denote by
$\boldsymbol{\gamma}_i$ the vector of moments $(\gamma_i(j))$, $j=0,1,\ldots$.

Let $(\mathscr{L}_j)$, $j=0,1,\ldots$, be the family of orthonormal polynomials with respect to the
Lebesgue measure on $[0,1]$. They can be deduced from the Legendre polynomials\footnote{The Legendre polynomials are
orthonormal w.r.t.\ to the Lebesgue measure on $[-1,1]$.} via the change of variable
$t'=(2t-1)$; the $(\L_j)$ are called the {\it shifted Legendre polynomials}. The polyno\-mials~$(\L_j)$ can be also computed exactly from the moments $\boldsymbol{\gamma}_0$ of the Lebesgue measure $dt$ on $[0,1]$
by computing some determinants of modif\/ied Hankel moment matrices. For instance, $\L_0=1$, and
\begin{gather*}
\L_1(t) = a \det \left(
\begin{matrix}
1  &  1/2    \\
1  & t
\end{matrix}
\right) = a (t-1/2) \qquad \mbox{with} \quad a^2  (1/3-1/2+1/4) = 1,
\end{gather*}
i.e., $a=\sqrt{12}$ and $\L_1(t)=2\sqrt{3} t-\sqrt{3}$, and
\begin{gather*}
\L_2(t) = b \det \left(
\begin{matrix}
1  &  1/2  &1/3  \\
1/2  & 1/3 & 1/4\\
1 & t & t^2
\end{matrix}
\right) = b \big[t^2/12-t/12+1/72\big],
\end{gather*}
with $b>0$ such that $\int_0^1\L_2(t)^2dt=1$. See, e.g., \cite{dunkl} and \cite{helton}.

The Lebesgue space $L_2([0,1]):=\{f\colon [0,1]\to\R,\, \int_0^1f^2\, dx<\infty\}$, equipped with the scalar product $\langle f,g\rangle=\int_0^1f g\,dx$ and the associated norm $\Vert\cdot\Vert$, is a Hilbert space and the polynomials are dense in $L_2([0,1])$.
In particular the family $(\L_j)$, $j=0,1,\ldots$, form an orthonormal basis of $L_2([0,1])$.
Let~$\ell^2$ denotes the space of square-summable sequences with norm also denoted by~$\Vert\cdot\Vert$.

Finally, let $\Vert f\Vert_\infty:={\rm ess}\sup\limits_{x\in [0,1]} \vert f(x)\vert$, and similarly,
for $p\in\R[x]$ let $\Vert p\Vert_\infty:=\sup\limits_{x\in [0,1]}\vert p(x)\vert$.
Then $L_\infty([0,1]):=\{f\colon \Vert f\Vert_\infty<\infty\}$.

\subsection{Some preliminary results}
We next state some useful auxiliary results, some of them being standard in Real Analysis.

\begin{Proposition}
\label{lem1}
Let $t\mapsto f(t)$ be an element of $L_2([0,1])$ and define $\f=(\f(j))$ by
\begin{gather*}
\f(j) := \int_0^1 \L_j(t) f(t)\, dt,\qquad j=0,1,\ldots.
\end{gather*}
Then one has
\begin{gather}
\label{L2-norm-conv}
\displaystyle\sum_{j=0}^\infty \f(j)\L_j \left(:= \lim_{n\to\infty} \sum_{j=0}^n\f(j) \L_j \right) = f \qquad\mbox{in} \quad L_2([0,1]),
\end{gather}
and this decomposition is unique. Moreover
$\f\in\ell^2$ and $\Vert f\Vert = \Vert \f\Vert$.
\end{Proposition}

When the interval is $[-1,1]$ (instead of $[0,1]$ here)~\eqref{L2-norm-conv} is called the Legendre (or Legendre--Fourier)  series expansion of the function~$f$ and $\f=(\f(j))$ is called the vector of Legendre--Fourier coef\/f\/icients.

The notation $f^k$ stands for the function $t\mapsto f(t)^k$, $k\in\N$.
If $f^k\in L_2([0,1])$ we denote
by $\f_k=(\f_k(j))\in\ell^2$ its (shifted) Legendre--Fourier coef\/f\/icients so that
$\Vert f^k\Vert =\Vert \f_k\Vert$ (where again the latter norm is that of $\ell^2$).
Notice that we also have:

\begin{Proposition}
\label{prop1}
Let $f,f^k,g\in L_2([0,1])$ with $k\in\N$. Then $g=f^k$ if and only if
$\hat{\boldsymbol{g}}=\hat{\boldsymbol{f}_k}$.
\end{Proposition}

This follows from the uniqueness  of the decomposition in the basis $(\L_j)$.

We also have the following helpful results:

\begin{Lemma}
\label{lem2}
Let $f\in L_2([0,1])$ with Legendre--Fourier coefficients $\f=(\f(j))$,
and let $(p_n)\subset\R[t]$
be a sequence of polynomials such that $\Vert p_n-f\Vert\to0$ as $n\to\infty$.

If $\hat{\boldsymbol{p}}_n=(\hat{\boldsymbol{p}}_n(j))$ denotes the $($shifted$)$ Legendre--Fourier coefficients of
$p_n$, for all $n=1,2,\ldots$, then  $\Vert\hat{\boldsymbol{p}}_n-\f\Vert\to0$ in $\ell^2$ as $n\to\infty$.
\end{Lemma}

\begin{proof}
As $\Vert p_n-f\Vert^2=\int_{0}^1(p_n-f)^2\, dt$ and with $d_n=\deg (p_n)$,
\begin{gather*}
\int_{0}^1(p_n-f)^2\, dt =
\int_{0}^1\left(\sum_{j=0}^{d_n}\hat{\boldsymbol{p}}_n(j) \L_j-f\right)^2\,dt\\
\hphantom{\int_{0}^1(p_n-f)^2\, dt}{}
 = \sum_{j=0}^{d_n}(\hat{\boldsymbol{p}}_n(j))^2 \underbrace{\int_{0}^1\L_j^2\,dt}_{=1}
+2\sum_{k<j}\hat{\boldsymbol{p}}_n(j)\hat{\boldsymbol{p}}_n(k) \underbrace{\int_{0}^1\L_j\L_k\,dt}_{=0}\\
\hphantom{\int_{0}^1(p_n-f)^2\, dt=}{}
 -2 \sum_{j=0}^{d_n}\hat{\boldsymbol{p}}_n(j) \underbrace{\int_{0}^1\L_j f\,dt}_{\f(j)}+\underbrace{\Vert f\Vert ^2}_{=\Vert\f\Vert= \sum_j\f(j)^2}\\
\hphantom{\int_{0}^1(p_n-f)^2\, dt}{}
\geq\sum_{j=0}^{d_n}\hat{\boldsymbol{p}}_n(j)^2-2\hat{\boldsymbol{p}}_n(j) \f(j)+\f(j)^2 =
\sum_{j=0}^{d_n}(\hat{\boldsymbol{p}}_n(j)-\f(j))^2,
\end{gather*}
and the result follows because $\Vert p_n-f\Vert\to0$ as $n\to\infty$.
\end{proof}

\begin{Lemma}
\label{lem3}
Let $f\in L_\infty([0,1])$ $($hence $f^k\in L_2([0,1])$ for every $k\in\N)$.
If a sequence $(p_n)\subset\R[x]$ is such that
$\sup_n\Vert p_n\Vert_\infty<\infty$ and $\Vert p_n-f\Vert\to0$ as $n\to\infty$, then for every $k\in\N$,
\begin{gather*}
\lim_{n\to\infty} \big\Vert p_n^k-f^k\big\Vert = 0.
\end{gather*}
In addition, if $\hat{\boldsymbol{p}}^k_n$ denotes the $($shifted$)$ Legendre--Fourier coefficients of
$p^k_n$, $n=1,\ldots$, then $\Vert\hat{\boldsymbol{p}}^k_n-\f_k\Vert\to0$ in $\ell^2$ as $n\to\infty$.
\end{Lemma}

\begin{proof}
Let $M>\max[\Vert f\Vert_\infty, \sup_n\Vert p_n\Vert_\infty]$ and f\/ix $k\in\N$,
\begin{gather*}
\big\Vert p^k_n-f^k\big\Vert^2 = \int_0^1\big(p_n^k-f^k\big)^2\,dx = \int_0^1(p_n-f)^2 \left(\sum_{\ell=0}^{k-1} p_n^{k-1-\ell}f^\ell \right)^2 dx\\
\hphantom{\big\Vert p^k_n-f^k\big\Vert^2}{}
\leq \int_0^1(p_n-f)^2\left(\sum_{\ell=1}^{k}\big\vert p_n^{k-\ell}f^{\ell-1}\big\vert\right)^2dx
\leq \big(kM^{k-1}\big)^2\int_0^1(p_n-f)^2\,dx\\
\hphantom{\big\Vert p^k_n-f^k\big\Vert^2}{}
=\big(kM^{k-1}\big)^2\Vert p_n-f\Vert^2 \quad (\to0 \ \mbox{as $n\to\infty$.})
\end{gather*}
Then the last statement follows from Lemma~\ref{lem2}.
\end{proof}

\begin{Definition}
\label{def-1}
Let $f\in L_2([0,1])$ with Legendre--Fourier coef\/f\/icients  $\f$. For every $k,n\in\N$, def\/ine
the polynomial $f^{(k)}_n\in\R[x]$ and the vector $\f^{(k)}_{n}\in\R^{kn+1}$ by
\begin{gather*}
t\mapsto f^{(k)}_n(t) := \left(\sum_{j=0}^n\f(j) \L_j(t)\right)^k = \sum_{j=0}^{nk}\f^{(k)}_n(j) \L_j(t).
\end{gather*}
Observe that each entry $\f^{(k)}_n(j)$, $j=0,\ldots,nk$, is a degree-$k$ form of the
f\/irst $n+1$ Legendre--Fourier coef\/f\/icients of $\hat{\boldsymbol{f}}$.
Completing with zeros, consider $\f^{(k)}_n$ to be an element of $\ell^2$ and if $\f^{(k)}_n$ converges in~$\ell^2$ as $n\to \infty$, call $\f^{(k)}\in\ell^2$ its limit.

The limit $\f^{(k)}$ can also be denoted $\f\star\cdots\star\f$, the limit of the $k$ times ``$\star$-product''
in~$\ell^2$ of the vector $\f\in\ell^2$ by itself. Equivalently one may write $\f^{(k)}=\f^{(k-1)}\star \f$ since
$f^{(k)}_n(t)=f^{(k-1)}_n(t)f^{(1)}_n(t)$ for all $t\in [0,1]$, and $\f^{(k-1)}_n\to \f^{k-1}$ as $n \to\infty$, as well as
$\f^{(1)}_n\to \f$.
\end{Definition}

\begin{Lemma}
\label{lem-funda}
Let $f\in L_\infty([0,1])$, hence $f^k\in L_2([0,1)$ with $($shifted$)$ Legendre--Fourier coefficients $\f_k\in\ell^2$ for every $k\in\N$, and assume that
\begin{gather*}
\sup_n \Vert f^{(1)}_n\Vert_\infty\ \left(=
\sup_n \Vert \sum_{j=0}^n\f(j) \L_j \Vert_\infty  \right) < \infty.
\end{gather*}
Then $\f_k=\f^{(k)}=\f\star\cdots\star\f$ $(k$ times$)$ for every $k=1,2,\ldots$, meaning that for every fixed $k\in\N$,
\begin{gather*}
\lim_{n\to\infty} \left\Vert \left(\sum_{j=0}^n\f(j) \L_j \right)^k-f^k \right\Vert \ \left(=
\lim_{n\to\infty} \left\Vert \sum_{j=0}^{kn}\f^{(k)}_n(j) \L_j -f^k \right\Vert\right)\\
\hphantom{\lim_{n\to\infty} \left\Vert \left(\sum_{j=0}^n\f(j) \L_j \right)^k-f^k \right\Vert}{}
\quad
 = \lim_{n\to\infty}\left\Vert \sum_{j=0}^{n}\f_k(j) \L_j -f^k \right\Vert = 0.
\end{gather*}
Equivalently, $\f_{k}=\f_{k-1}\star \f$ for every $k=2,3,\ldots$.
\end{Lemma}

\begin{proof}
The result  is a direct consequence of Lemmas~\ref{lem2} and~\ref{lem3} with $p_n=f^{(1)}_n$
(and the def\/inition of the limit ``$\star$-product'' in Def\/inition \ref{def-1}).
\end{proof}

\section{Main result}

Assume that we are given all moments of a nonnegative measure $d\mu(x,t)$ on
a box $[a,b]\times [c,d]\subset\R^2$. After a re-scaling of its moments
we may and will assume that $\mu$ is a probability measure supported on $[0,1]^2$ with associated moments
\begin{gather*}
\gamma_i(j) = \int_{[0,1]^2}x^i t^j\,d\mu(x,t),\qquad i,j=0,1,\ldots.
\end{gather*}
We further assume that the marginal measure $\mu_t$ with respect to the variable~$t$, is the Lebesgue measure
on $[0,1]$, that is, $\gamma_0(j)=1/(j+1)$, $j=0,1,\ldots$.

A standard disintegration of the measure $\mu$ yields
\begin{gather}\label{aux44}
\gamma_i(j) = \int_{[0,1]^2} t^j x^i\,d\mu(x,t) = \int_0^1t^j  \Bigg(\underbrace{\int_{[0,1]}x^i  \psi(dx\vert t)}_{=: f_i(t)}\Bigg) dt,\qquad i,j=0,1,\ldots,
\end{gather}
where the stochastic kernel $\psi(\cdot\vert\,t)$ is the conditional probability on $[0,1]$ given $t\in [0,1]$. Observe that the measurable function~$f_i$ in~\eqref{aux44} is nonnegative and uniformly bounded by $1$ because $\vert x^i \vert\leq 1$ on $[0,1]$ for every~$i$, and so $f_i\in L_\infty([0,1])$ for every $i=1,\ldots$.

The vector $\boldsymbol{\gamma}_i=(\gamma_i(j))$, $j=0,1,\ldots$, is the vector of moments of the measure
$d\mu_i(t)=f_i(t)dt$ on $[0,1]$, for every $i=1,2,\ldots$. The (shifted) Legendre--Fourier vector of coef\/f\/icients
$\boldsymbol{\hat{q}}_i$ of $f_i$ are obtained
easily from the (inf\/inite) vector $\boldsymbol{\gamma}_i$ via a triangular linear system. Indeed write
\begin{gather*}
\L_j(t)=\sum_{k=0}^j\Delta_{jk}t^k,\qquad \forall \,t\in [0,1],\quad j=0,1,\ldots,
\end{gather*}
where $\Delta_{jj}>0$, or in compact matrix form
\begin{gather}\label{matrix-delta}
\left[\begin{matrix} \L_0\\ \L_1\\ \cdot\\ \L_n\\\cdot\end{matrix}\right] = \boldsymbol{\Delta}
\left[\begin{matrix} 1\\ t\\ \cdot\\ t^n\\\cdot\end{matrix}\right],
\end{gather}
for some inf\/inite lower triangular matrix $\Delta$ with all diagonal elements being strictly positive. Therefore
\begin{gather*}
\boldsymbol{\Delta} \boldsymbol{\gamma}_i = \boldsymbol{\Delta} \int_0^1\left[\begin{matrix} 1\\ t\\ \cdot\\ t^n\\\cdot\end{matrix}\right] f_i(t)\,dt = \int_0^1 \left[\begin{matrix} \L_0\\ \L_1\\ \cdot\\ \L_n\\\cdot\end{matrix}\right] f_i(t)\,dt =\hat{\boldsymbol{q}}_i.
\end{gather*}

Suppose that the measure $\mu$ is supported on a trajectory
$\{(t,x(t))\colon t\in [0,1]\}\subset [0,1]^2$ for some measurable (density) function $x\colon [0,1]\to [0,1]$.
The measurable function $t\mapsto x(t)$ is an element of $L_\infty([0,1])$ because $\Vert x\Vert_\infty\leq 1$.
Then by Proposition \ref{lem1},
\begin{gather}\label{x(t)}
x = \sum_{j=0}^\infty \x(j)  \L_j\qquad \mbox{in $L_2([0,1])$},
\end{gather}
where $\x=(\x(j))\in\ell^2$ is its vector of (shifted) Legendre--Fourier coef\/f\/icients
(with $\Vert\x\Vert=\Vert x\Vert$).
Similarly, for every $k=2,3,\ldots$, the function $t\mapsto x(t)^k$ is in $L_\infty([0,1])$ and
\begin{gather*}
x^k= \sum_{j=0}^\infty \x_k(j) \L_j\qquad \mbox{in $L_2([0,1])$},
\end{gather*}
with vector of (shifted) Legendre--Fourier coef\/f\/icients $\x_k\in\ell^2$ such that $\Vert x^k\Vert=\Vert\x_k\Vert$.

We also recall the notation $\x^{(k)}_n\in\R^{kn+1}$ for the vector of coef\/f\/icients in the basis $(\L_j)$
of the polynomial $t\mapsto\left( \sum\limits_{j=0}^n\x(j)\L_j(t)\right)^k$, and when considered as an element
of $\ell^2$ (by completing with zeros) denote by $\x^{(k)}\in\ell^2$ its limit when it exists.

\begin{Theorem}
\label{thmain}
Let $\mu$ be a Borel probability measure on $[0,1]^2$ and let $\boldsymbol{\gamma}_i(j)$, $i,j=0,1,\ldots$, be the moments of~$\mu$ in~\eqref{def-moment}.

\begin{enumerate}\itemsep=0pt
\item[$(a)$] If  $\mu$ is supported on a trajectory $\{(t,x(t))\colon t\in [0,1]\}$ for some nonnegative measurable function $t\mapsto x(t)$
on $[0,1]$ and if $\sup\limits_n\left\Vert\sum\limits_{j=0}^n\x(j)\L_j\right\Vert_\infty <\infty$, then
\begin{gather}\label{thmain-1}
\hat{\boldsymbol{x}}_i = \x^{(i)} = \underbrace{\x\star\cdots\star\x}_{\mbox{$i$ times}} =
\hat{\boldsymbol{x}}_{i-1}\star\hat{\boldsymbol{x}},\qquad \forall\,i=2,3,\ldots,
\end{gather}
Equivalently
\begin{gather}
\label{thmain-2}
\boldsymbol{\Delta} \boldsymbol{\gamma}_i = \left(\boldsymbol{\Delta} \boldsymbol{\gamma}_1\right)^{(i)} =
\left(\boldsymbol{\Delta} \boldsymbol{\gamma}_{i-1}\star\boldsymbol{\Delta} \boldsymbol{\gamma}_1\right)^{(i)},
\qquad \forall\, i=2,3,\ldots,
\end{gather}
where $\boldsymbol{\Delta}$ is the non singular triangular matrix defined in~\eqref{matrix-delta}.

\item[$(b)$] Conversely, if~\eqref{thmain-2} holds then~$\mu$ is supported on a trajectory
$\{(t,x(t))\colon t\in [0,1]\}$ for some measurable function $t\mapsto x(t)$ on $[0,1]$, and~\eqref{thmain-1} also holds.
\end{enumerate}
\end{Theorem}

\begin{proof}
The (a) part. As $\mu$ is supported on $[0,1]^2$ one has $\Vert x\Vert_\infty\leq 1$
and so the function $t\mapsto x(t)^i$ is in $L_2([0,1])$ for every $i=1,2,\ldots$.
So let $t\mapsto x(t)$ be written as in~\eqref{x(t)}. Consider
the function $t\mapsto x(t)^i$, for every f\/ixed $i\in\N$, so that
\begin{gather*}
x^i = \sum_{j=0}^\infty \x_i(j)  \L_j \qquad\mbox{in $L_2([0,1])$},
\end{gather*}
where the (shifted) Legendre--Fourier vector of coef\/f\/icients $\x_i$ is  obtained by
$\x_i=\boldsymbol{\Delta}^{-1}\boldsymbol{\gamma}_i$. But by Lemma~\ref{lem-funda},
we also have
\begin{gather*}
x^i = \left(\sum_{j=0}^\infty\x(j) \L_j\right)^i  \ \left(= \lim_{n\to\infty}\left(\sum_{j=0}^n \x(j) \L_j \right)^i \right)\qquad \mbox{in $L_2([0,1])$},
\end{gather*}
with $\Vert x^i\Vert=\Vert\x_i\Vert$ and $\x_i=\boldsymbol{\Delta} \boldsymbol{\gamma}_i$.
In other words, $\x^{(i)}=\x_i$ or equivalently,
$\hat{\boldsymbol{x}}^{(i)} = \boldsymbol{\Delta} \boldsymbol{\gamma}_i=
\left(\boldsymbol{\Delta} \boldsymbol{\gamma}_1\right)^{(i)}$,
which is~\eqref{thmain-1}.

We next prove the (b) part.
By the disintegration \eqref{aux44} of the measure $\mu$,
\begin{gather*}
\gamma_i(j) = \int_0^1 t^j f_i(t)\,dt, \qquad i,j=0,1,\ldots
\end{gather*}
for some nonnegative measurable functions $f_i\in L_\infty([0,1])$, $i=1,2,\ldots$.

As~\eqref{thmain-2} holds one may conclude that
$\hat{\boldsymbol{q}}_i=\hat{\boldsymbol{q}}_1^{(i)}$ where $\hat{\boldsymbol{q}}_i$
is the (shifted) Legendre--Fourier vector of coef\/f\/icients associated with $f_i\in L_2([0,1])$, $i=1,\ldots$.
Hence by Proposition~\ref{prop1}, $f_i(t)=f_1(t)^i$ a.e.\ on~$[0,1]$, for every $i=1,2,\ldots$. That is, for every $i=1,2,\ldots$,
there exists a~Borel set $B_i\subset [0,1]$ with Lebesgue measure zero such that
$f_i(t)=f_1(t)^i$ for all $t\in [0,1]{\setminus} B_i$. Therefore  the Borel set $B=\cup_{i=1}^\infty B_i$ has Lebesgue measure zero and for all $i=1,2,\ldots$,
\begin{gather*}
f_i(t)=f_1(t)^i,\qquad\forall\, t\in [0,1]{\setminus} B.
\end{gather*}
Hence for every $t\in [0,1]{\setminus} B$,
\begin{gather*}
\int_{[0,1]}x^i\psi(dx\vert t)=f_1(t)^i=\int_{[0,1]} x^i\,d\delta_{f_1(t)},\qquad \forall\, i=1,2,\ldots.
\end{gather*}
where $\delta_{f_1(t)}$ is the Dirac measure at the point $f_1(t)\in [0,1]$.
As measures on compact sets are moment determinate, one must have $\psi(dx\vert t)=\delta_{f_1(t)}$, for all
$t\in [0,1]{\setminus} B$. Therefore $d\mu(x,t)=\delta_{f_1(t)}\,dt$, i.e., the measure $\mu$ is supported on the trajectory $\{(t,x(t))\colon t\in [0,1]\}$, where $x(t)=f_1(t)$ for
almost all $t\in [0,1]$.
\end{proof}

\begin{Remark}
If the trajectory $t\mapsto x(t)$ is a polynomial of degree say $d$, then the vector of Legendre--Fourier coef\/f\/icients $\x\in\ell^2$ has at most
$d+1$ non-zero elements. Therefore for every $j=2,\ldots$, $\x_j\in\ell^2$ also has at most $jd+1$ non-zero elements
and the condition~\eqref{thmain-2} can be checked easily.
\end{Remark}

In Theorem \ref{thmain}(a) one assume that
$\sup\limits_n\left\Vert \sum\limits_{j=0}^n \x(j) \L_j\right\Vert_\infty<\infty$ which is much weaker than, e.g., assuming
the uniform convergence $\left\Vert \sum\limits_{j=0}^n \x(j) \L_j-x\right\Vert_\infty\to 0$  as $n\to\infty$.
The latter (which is also much stronger than the a.e.\ pointwise convergence)
can be obtained if the function~$x(t)$ has some smoothness properties.
For instance if~$x$ belongs to some Lipschitz class of order larger then or equal to $1/2$, then uniform convergence takes place and one may even obtain rates of convergence; see, e.g.,~\cite{suetin} and also~\cite{wang} for a comparison (in terms of convergence) of Legendre and Chebyshev expansions.
In fact, quoting the authors of \cite{gottlieb1}, {\em ``\dots\ knowledge of the partial spectral sum of an~$L_2$ function in $[-1,1]$ furnishes enough information such that an exponential convergent approximation can be constructed in any subinterval in which $f$ is analytic''.}

\begin{Example}
To illustrate Theorem~\ref{thmain} consider the following toy example with $\mu$ on $[0,1]^2$ and with marginal w.r.t.~``$t$'' being
the uniform distribution on $[0,1]$ and conditional $\psi(dx\vert t)=\delta_{\exp(-t)}$ for all $t\in [0,1]$. That is,
$t\mapsto x(t)=\exp(-t)$.

Then the f\/irst $11$ Legendre--Fourier coef\/f\/icients $\x(j)$, $j=0,\ldots,10$ of $x$ read
\begin{gather*}\x =[\begin{matrix}0.63212055 & -0.1795068  & 0.0230105&  -0.0019370&    0.0001217 \end{matrix} \\
\hphantom{\x =[}{}
\begin{matrix} -0.0000061 &  0.0000002&
  -0.00000001 &  -0.00000004&  -0.0000015&  -0.0000625\end{matrix}].
\end{gather*}
  Similarly the f\/irst $11$ Legendre--Fourier coef\/f\/icients of $t\mapsto x(t)^2=\exp(-2t)$ read
\begin{gather*}\x_2=[ \begin{matrix} 0.4323323 &  -0.2344075 &  0.0588678&  -0.0097965&   0.0012219\end{matrix} \\
\hphantom{\x_2=[}{} \begin{matrix} -0.0001219& 0.0000101& -0.0000007 &   0.00000004& -0.000000004& -0.00000004\end{matrix}].
\end{gather*}
With $n=5$ the polynomial $t\mapsto x^{(2)}_5(t):=\left(\sum\limits_{k=0}^5\x(j) \L_j(t)\right)^2$ reads
 \begin{gather*}
 x^{(2)}_5(t) =\sum_{k=0}^{10} \x_5^{(2)}(j)\,\L_j(t),\qquad t\in\R,\qquad\mbox{with}\\
 \x_5^{(2)}=[\begin{matrix}
0.4323336&  -0.2344129&   0.0588626&  -0.0097976&   0.0012219\end{matrix}\\
\hphantom{\x_5^{(2)}=[}{}
\begin{matrix}  -0.0001218& 0.0000098 & -0.0000006&  0.00000003& -0.000000001&  0.0000000\end{matrix}],
\end{gather*}
and we can observe that $\x^{(2)}_5-\x_2\approx O\big(10^{-5}\big)$.

In fact the curves
$t\mapsto x^{(2)}_5(t)$ and $t\mapsto \exp(-2t)$ are almost indistinguishable on the inter\-val~$[0,1]$.
\end{Example}

\subsection{A more general case}

We have considered a measure $\mu$ on $[0,1]^2$ whose marginal with respect to $t\in [0,1]$ is the Lebesgue measure.
The conditions of Theorem \ref{thmain} are naturally stated in terms of the (shifted) Legendre--Fourier coef\/f\/icients
associated with the functions $t\mapsto f_i(t)$ of $L_2([0,1])$ def\/ined in~\eqref{aux44}.

However, the same conclusions also hold
if the marginal of $\mu$ with respect to $t\in [0,1]$ is some measure $d\nu=h(t)dt$ for some nonnegative function $h\in L_1([0,1])$
with all moments f\/inite.
The only change is that now we have to consider the orthonormal polynomials
$t\mapsto \mathcal{H}_j(t)$, $j=0,1,\ldots$, with respect to~$\nu$. Recall that all the $\mathcal{H}_j$'s can be computed from the moments
\begin{gather*}
\gamma_0(j)=\int t^j d\mu(x,t)=\int_0^1 t^j\,d\nu(t)=\int_0^1 t^j h(t)\,dt,\qquad j=0,1,\ldots.
\end{gather*}
Then proceeding as before, for every $i=1,2,\ldots$,
\begin{gather*}
\int t^j x^i\,d\mu(x,t) = \int_0^1t^j \Bigg(\underbrace{\int_{[0,1]}x^i\psi(dx\vert t)}_{=f_i(t)\in L_2([0,1],\nu)}\Bigg) h(t)\,dt,\qquad j=0,1,\ldots,
\end{gather*}
and we now consider the vector of coef\/f\/icients $\hat{\boldsymbol{f}}_{hi}=(\hat{\boldsymbol{f}}_{hi}(j))$ def\/ined by
\begin{gather*}
\hat{\boldsymbol{f}}_{hi}(j) = \int_0^1 \mathcal{H}_j(t) f_i(t) h(t)\,dt,\qquad j=0,1,\ldots
\end{gather*}
the analogues for the measure $d\nu= h(t)dt$ and the function $f_i\in L_2([0,1],\nu)$, of the (shifted) Legendre--Fourier coef\/f\/icients $\boldsymbol{\hat{f}_i}(j)$ of $f_i$ in~\eqref{aux44} for the Lebesgue measure on~$[0,1]$.
Then the conditions in Theorem~\ref{thmain} would be exactly the same as before, excepted that now,
$\x=(\x(j))$ with
\begin{gather*}
\x(j)=\int_0^1 \mathcal{H}_j(t) x(t)h(t)\,dt,\qquad j=0,1,\ldots.
\end{gather*}

\subsection{Discussion}

Theorem \ref{thmain} may have some practical implications. For instance
consider the \textit{weak} formula\-tion~$\mathcal{P}$ of an optimal control problem~$\P$ as an
inf\/inite-dimensional optimization problem on an appropriate space of (occupation) measures, as
described, e.g., in~\cite{vinter}. In~\cite{sicon} the authors propose to solve a hierarchy of semidef\/inite
relaxations $(\mathcal{P}_k)$, $k=1,2,\ldots$, of~$\mathcal{P}$.  Each optimal solution
of $\mathcal{P}_k$ provides with a f\/inite sequence $\z^k=(z^k_{j,\alpha,\beta})$ such that when $k\to\infty$,
$\z^k\to \z^*$ where~$\z^*$ is the inf\/inite sequence of some measure $d\mu(t,\bx,\u)$ on
$[0,1]\times \X\times \U$, where $\X\subset\R^n$, $\U\subset\R^m$, are compact sets.

Under some conditions both problems $\P$ and its relaxation $\mathcal{P}$ have same optimal value.
If $\mu$ is supported on feasible trajectories $\{(t,\bx(t),\u(t))\colon t\in [0,1]\}$ then these trajectories
are optimal for the initial optimal control problem~$\P$. So it is highly desirable to check whether
indeed~$\mu$ is supported on trajectories from the only knowledge of its moments $\z^*=(z^*_{j,\alpha,\beta})$.

By construction of the moment sequences $\z^k$ one already knows that
the marginal of $\mu$ with respect to the variable ``$t$'' is the Lebesgue measure on~$[0,1]$. Therefore we are typically in
the situation described in the present paper. Indeed to check
whether $\mu$ is supported on trajectories $\{(t,(x_1(t),\ldots,x_n(t),u_1(t),\ldots,u_m(t))\colon t\in [0,1]\}$, one considers each coordinate~$x_i(t)$ or $u_j(t)$ separately. For instance, for $x_i(t)$ one considers the subset of moments $\gamma_k(j)=(z^*_{j,\alpha,0})$
with $j=0,1,\ldots$, $\alpha=(0,\ldots,0,k,0,\ldots,0)\in\N^n$, $k=0,1,\ldots$, with $k$ in position~$i$.
If~\eqref{thmain-1} holds  then indeed the marginal~$\mu_{t,x_i}$ of~$\mu$ on $(t,x_i)$, with moments~$(\gamma_k(j))$
is supported on a trajectory $\{(t,x_i(t))\colon t\in [0,1]\}$.

Of course, in~\eqref{thmain-1} there are countably many conditions to check whereas
in principle only f\/initely many moments of~$z^*$ are available (and with some inaccuracy
due to (a) solving numerically a truncation $\mathcal{P}_k$ of $\mathcal{P}$, and (b) the convergence $\z^k\to\z^*$ has not  taken place yet). So an issue of future investigation is to provide necessary (or suf\/f\/icient?)
conditions based only on f\/initely many (approximate) moments of~$\mu$.

\subsection*{Acknowledgements}

Research  funded by the European Research Council
(ERC) under the European Union's Horizon 2020 research and innovation program
(grant agreement ERC-ADG 666981 TAMING).

\pdfbookmark[1]{References}{ref}
\LastPageEnding

\end{document}